\documentclass{amsproc}
\usepackage{graphicx}
\usepackage{amssymb}
\usepackage{epstopdf}
\usepackage{amsmath,amsthm,amscd,amssymb}
\usepackage{latexsym}
\usepackage[colorlinks,citecolor=red,pagebackref,hypertexnames=false]{hyperref}
\usepackage{geometry}                
\numberwithin{equation}{section}

\theoremstyle{plain}
\newtheorem{theorem}{Theorem}[section]
\newtheorem{lemma}[theorem]{Lemma}

\theoremstyle{definition}

\newtheorem{case[theorem]}{Case}

\theoremstyle{remark}

\numberwithin{equation}{section}

\begin{document}

\title{\parbox{14cm}{\centering{Orthogonal systems in vector spaces over finite fields}}}


\author{Alex Iosevich and Steve Senger}

\address{Department of Mathematics, University of Missouri, Columbia, MO 65211-4100}

\email{iosevich@math.missouri.edu}

\email{senger@math.missouri.edu}

\thanks{A. Iosevich was supported by the NSF Grant DMS04-56306 and S. Senger was supported by the NSF Grant DMS07-04216}

\begin{abstract} We prove that if a subset of the $d$-dimensional vector space over a finite field is large enough, then it contains many $k$-tuples of mutually orthogonal vectors. \end{abstract}

\maketitle

\tableofcontents

\section{Introduction}

A classical set of problems in combinatorial geometry deals with the question of whether a sufficiently large subset of ${\Bbb R}^d$, ${\Bbb Z}^d$, or ${\Bbb F}_q^d$ contains a given geometric configuration. 
For example, a classical result due to Furstenberg, Katznelson and Weiss (\cite{FKW90}; see also \cite{B86}) says that if $E \subset {\mathbb R}^2$ has positive upper Lebesgue density, then for any $\delta>0$, the $\delta$-neighborhood of $E$ contains a congruent copy of a sufficiently large dilate of every three point configuration.

When the size of the point set is smaller than the dimension of ambient Euclidean space, taking a $\delta$-neighborhood is not necessary, as shown by Bourgain in \cite{B86}. He proves that if $E \subset {\Bbb R}^d$ has positive upper density and $\Delta$ is a $k$-simplex with $k<d$, then $E$ contains a rotated and translated image of every large dilate of $\Delta$. The case $k=d$ and $k=d+1$ remain open, however.  See also, for example, \cite{Berg96}, \cite{F81}, \cite{K07}, \cite{TV06} and \cite{Z99} on related problems and their connections with discrete analogs.

In the geometry of the integer lattice ${\Bbb Z}^d$, related problems have been recently investigated by Akos Magyar in \cite{M06} and \cite{M07}. In particular, he proves in \cite{M07} that if $d>2k+4$ and $E \subset {\Bbb Z}^d$ has positive upper density, then all large (depending on density of $E$) dilates of a $k$-simplex in ${\Bbb Z}^d$ can be embedded in $E$. Once again, serious difficulties arise when the size of the simplex is sufficiently large with respect to the ambient dimension.

In finite field geometries a step in this direction was taken by the second and third listed authors in \cite{HI07}. They prove that if $E \subset {\Bbb F}_q^d$, the $d$-dimensional vector space over the finite field with $q$ elements with $|E| \ge Cq^{d \frac{k-1}{k}+\frac{k-1}{2}}$ and $\Delta$ is a $k$-dimensional simplex, then there exists $\tau \in {\Bbb F}_q^d$ and $O \in SO_d({\Bbb F}_q)$ such that $\tau+O(\Delta) \subset E$. The result is only non-trivial in the range $d \ge \left(^k_2\right)$ as larger simplexes are out of range of the methods used. See also \cite{HIKSU08} for a thorough graph theoretic analysis of a more general problem.

In this paper we ask whether a sufficiently large subset of ${\Bbb F}_q^d$, the $d$-dimensional vector space over the finite field with $q$ elements, contains a $k$-tuple of mutually orthogonal vectors. Similar questions, at least in the context of pairs of orthogonal vectors are studied in \cite{AK97}. This problem does not have a direct analog in Euclidean or integer geometries because placing the set strictly inside $\{x \in {\Bbb R}^d: x_j>0\}$ immediately guarantees that no orthogonal vectors are present. However, the arithmetic of finite fields allows for a richer orthogonal structure. Our main result is the following. 
\begin{theorem} \label{main}
Let $E \subset \mathbb{F}_q^d$, such that 
$$ \vert E \vert \ge Cq^{d\frac{k-1}{k}+\frac{k-1}{2}+\frac{1}{k}}$$ with a sufficiently large constant $C>0$, where 
$$0<(^k_2) < d.$$

Let $\lambda_k$ be the number of $k$-tuples of $k$ mutually orthogonal vectors in $E$. Then 
$$\lambda_k=(1+o(1)) \vert E \vert^k q^{-\left(^k_2\right)}.$$
\end{theorem}

\vskip.125in

\subsection{Graph theoretic interpretation} Define a hyper-graph $G_k(q,d)$ by taking its vertices to be the elements of ${\Bbb F}_q^d$ and connect $k$ vertices by a hyper-edge if they are mutually orthogonal. Theorem \ref{main} above implies that any subgraph of $G_k(q,d)$ with more than $Cq^{d\frac{k-1}{k}+\frac{k-1}{2}+\frac{1}{k}}$ vertices contains $(1+o(1)){|E|}^k q^{-{k \choose 2}}$ hyper-edges, which is the statistically expected number.

Alternatively, we can think of Theorem \ref{main} as saying that any sub-graph of $G_2(q,d)$ of size greater than $Cq^{d\frac{k-1}{k}+\frac{k-1}{2}+\frac{1}{k}}$ contains $(1+o(1)){|E|}^k q^{-{k \choose 2}}$ complete sub-graph on $k$ vertices, once again a statistically expected number.

See \cite{HIKSU08}, and the references contained therein, for a systematic description of the properties of related graphs.

\vskip.125in

\subsection{Hyperplane discrepancy problem}

One of the key features of the proof of this result is the analysis of the following discrepancy problem. Let $$H_{x^1, x^2, \dots, x^k}=\{y \in {\Bbb F}_q^d: y \cdot x^j=0 \ j=1,2 \dots, k\}.$$

Define the discrepancy function $r_k$ by the equation 
$$ |E \cap H_{x^1, \dots, x^k}|=|E|q^{-k}+r_k(x^1, \dots, x^k),$$ where the first term should be viewed as the "expected" size of the intersection. In Lemma \ref{discrepancy} below we show that on average, 
$$ |r_k(x^1, \dots, x^k)| \lesssim \sqrt{|E|q^{-k}},$$ where here, and throughout the paper, $X \lesssim Y$ means that there exists $C>0$, independent of $q$, such that $X \leq CY$.

\subsection{Acknowledgements} The authors wish to thank Boris Bukh, Seva Lev and Michael Krivelevich for interesting comments and conversations pertaining to this paper.

\vskip.125in

\section{Proof of Theorem \ref{main}}

\vskip.125in

Observe that 
$$ r_{k-1}\left(x^1, ..., x^{k-1}\right) = q^{-(k-1)} \sum_{\substack{s_i \in \mathbb{F}_q^* \\ i=1,2,..., k-1}} \sum_{\substack{x^k \in \mathbb{F}_q^d}}E(x^k)\prod_{i=1}^{k-1}{\chi(-s_i x^i \cdot x^k)}. $$

\begin{lemma}\label{discrepancy}
$\Vert r_{k-1} \Vert_{L^2} \lesssim \vert E \vert ^\frac{1}{2} q^\frac{(d-1)(k-1)}{2}$.
\end{lemma}

Assuming Lemma \ref{discrepancy} for now, we prove the main result, Theorem \ref{main}.

\begin{proof}
Define $\mathcal{D}_k := \left\lbrace  (x^1, ..., x^k) \in E^k : x^i \cdot x^j = 0, \forall 1 \leq i < j \leq k \right\rbrace $, where $E^k$ means $\underbrace{E \times E \times ... \times E}_{k \text{ times}}$. Also, let $\mathcal{D}_k(x^1, ..., x^k)$ and $E(x)$ be the indicator functions for the set $\mathcal{D}_k$ and $E$, respectively. Clearly $\vert \mathcal{D}_k \vert = \lambda_k$.

\begin{align*}
\lambda_k &= \sum_{\substack{x^j \in \mathbb{F}_q^d : x^j \cdot x^k = 0 \\ j=1,2,...,k}}{\mathcal{D}_{k-1}(x^1, ..., x^{k-1})E(x^k)} \\
&= q^{-(k-1)} \sum_{\substack{x^j \in \mathbb{F}_q^d \\ j=1,2,..., k}}{\mathcal{D}_{k-1}(x^1, ..., x^{k-1})E(x^k)\sum_{\substack{s_i \in \mathbb{F}_q \\ i=1,2,..., k-1}}\prod_{i=1}^{k-1}{\chi(-s_i x^i \cdot x^k)}}\\
&= q^{-(k-1)} \sum_{\substack{s_i \in \mathbb{F}_q \\ i=1,2,..., k-1}}\sum_{\substack{x^j \in \mathbb{F}_q^d \\ j=1,2,..., k}}{\mathcal{D}_{k-1}(x^1, ..., x^{k-1})E(x^k)\prod_{i=1}^{k-1}{\chi(-s_i x^i \cdot x^k)}}\\
&= I + II + III,\\
\end{align*}

where we seperate the sum into three parts depending on the $s_i$'s. $I$ is the sum when all of the $s_i$'s are zero. $II$ is the sum when none of the $s_i$'s are equal to zero. $III$ is the sum when some of the $s_i$'s are equal to zero, and some are not. We treat these three cases seperately. We will show that $I$ dominates the other terms when $\vert E \vert$ satisfies the size condition, and is therefore the number of sets of $k$ mutually orthogonal vectors present in $E$, modulo a constant.

\begin{align*}
I &= \sum_{\substack{x^j \in \mathbb{F}_q^d \\ j=1,2,..., k}}{\mathcal{D}_{k-1}(x^1, ..., x^{k-1}) E(x^k) q^{-(k-1)}}\\
&= \sum_{\substack{x^j \in \mathbb{F}_q^d \\ j=1,2,..., k-1}}{\mathcal{D}_{k-1}(x^1, ..., x^{k-1}) \vert E \vert q^{-(k-1)}}\\
&= \vert E \vert q^{-(k-1)}\sum_{\substack{x^j \in \mathbb{F}_q^d \\ j=1,2,..., k-1}}{\mathcal{D}_{k-1}(x^1, ..., x^{k-1})}\\
&= \vert E \vert q^{-(k-1)} \lambda_{k-1}
\end{align*}

If $I$ indeed dominates the other two terms, we'll have
$$
\frac{\lambda_k}{\lambda_{k-1}} = \vert E \vert q^{-(k-1)}.
$$

To get an expression for $\lambda_k$, we run the computation for $k=2$ first. In this case, instead of $\mathcal{D}_1\left(x^1\right)$, we'll just have $E\left(x^1\right)$. This is because we are checking that any two orthogonal vectors in $\mathbb{F}_q^d$ are also in our set $E$. Running through the above calculation in this manner yields $\lambda_2 = \vert E \vert^2 q^{-1}.$

$$
\lambda_k = \frac{\vert E \vert^k}{q^{(^k_2)}}.
$$

Now we need to compute $II$, the biggest error term. Now we recall the definition of the discrepancy function.

$$
r_{k-1}\left(x^1, ..., x^{k-1}\right) = q^{-(k-1)}\sum_{\substack{s_i \in \mathbb{F}_q^* \\ i=1,2,..., k-1}} \sum_{\substack{x^k \in \mathbb{F}_q^d}}E(x^k)\prod_{i=1}^{k-1}{\chi(-s_i x^i \cdot x^k)}
$$

First, we dominate the sum of the $x^j$'s in $E$ by the sum of $x^j$'s in all of $\mathbb{F}_q^d$. Then we apply Cauchy-Schwarz to the sum over the $x^j$'s.

\begin{align*}
II &= q^{-(k-1)} \sum_{\substack{s_i \in \mathbb{F}_q^* \\ i=1,2,..., k-1}}\sum_{\substack{x^j \in \mathbb{F}_q^d \\ j=1,2,..., k}}{\mathcal{D}_{k-1}(x^1, ..., x^{k-1})E(x^k)\prod_{i=1}^{k-1}{\chi(-s_i x^i \cdot x^k)}}\\
&\leq \sum_{\substack{x^j \in E \\ j=1,2,..., k-1}} q^{-(k-1)} \sum_{\substack{s_i \in \mathbb{F}_q^* \\ i=1,2,..., k-1}}\sum_{\substack{x^k \in \mathbb{F}_q^d}}{\mathcal{D}_{k-1}(x^1, ..., x^{k-1})E(x^k)\prod_{i=1}^{k-1}{\chi(-s_i x^i \cdot x^k)}}\\
&\leq \lambda_{k-1}^\frac{1}{2} (\sum_{x^1, ..., x^{k-1}} r_{k-1}^2)^\frac{1}{2} \approx \vert E \vert^\frac{k-1}{2} q^\frac{-\left(^{k-1}_{\;\; 2}\right)}{2} q^{k-1} \approx \vert E \vert^\frac{k-1}{2}q^\frac{-\left(^k_2\right)}{2} \Vert r_{k-1}\Vert_{L^2}.
\end{align*}

So we use Lemma \ref{discrepancy} to get a handle on $\Vert r_{k-1} \Vert_{L^2}$. Now we are guaranteed that

\begin{align*}
II &\lesssim \vert E \vert^\frac{k-1}{2}q^\frac{-\left(^k_2\right)}{2} \vert E \vert^\frac{1}{2} q^\frac{(d-1)(k-1)}{2}. \\
&= \vert E \vert^\frac{k}{2} q^\frac{(d-1)(k-1)-\left(^k_2\right)}{2}
\end{align*}

To deal with $III$, break it up into sums that have the same number of non-zero $s_j$'s.

\begin{align*}
III &= \sum_{\text{one } s_j = 0} + \sum_{\text{two } s_j \text{'s} = 0} + ...\\
&= d\sum_{s_1 = 0} + d(d-1)\sum_{s_1 = s_2 = 0} + ...
\end{align*}

Now each of these sums will look like $II$, but with $(k-2)$ instead of $(k-1)$ for the first sum, and $(k-3)$ instead of $(k-1)$ in the second sum, and so on. This allows us to bound each sum in $III$ as follows:

$$
III \lesssim d\vert E \vert^\frac{k-2}{2}q^\frac{\left(^{k-1}_{\;\;\,2}\right)}{2}\Vert r_{k-2} \Vert_{L^2} + d(d-1)\vert E \vert^\frac{k-3}{2}q^\frac{\left(^{k-2}_{\;\;\,2}\right)}{2}\Vert r_{k-3} \Vert_{L^2} + ...
$$

So $III$ is dominated by $II$ as long as $q > d$, which is guaranteed, as $q$ grows arbirtarily large.

Now we only need to find appropriate conditions on $E$ to ensure that $I > II$.

\begin{align*}
I &> II \\
\vert E \vert^k q^{-\left(^k_2\right)} &> \vert E \vert^\frac{k}{2} q^\frac{(d-1)(k-1)-\left(^k_2\right)}{2} \\
\vert E \vert^\frac{k}{2} &> q^\frac{k(k-1)-2(d-1)(k-1)}{4} \\
\vert E \vert &> q^{\left(\frac{k-1}{k}\right)d +\frac{k-1}{2} + \frac{1}{k}}.
\end{align*}
\end{proof}

Now to prove the Lemma \ref{discrepancy}.

\begin{proof}
\indent Recall the definition of $r_{k-1}\left(x^1, ..., x^{k-1}\right)$ and use orthogonality in $x^1, ..., x^{k-1}$.
\begin{align*}
\Vert r_{k-1} \Vert_{L^2}^2 &= q^{-2(k-1)} \sum_{x^1, ..., x^{k-1}} \sum_{\substack{s_1, s_1', ..., \\s_{k-1}, s_{k-1}'}} \sum_{x^k, y^k \in E}  \prod_{j=1}^{k-1} \chi((s_j x^k - s_j' y^k) \cdot x^j) \\
&= q^{d(k-1)}q^{-2(k-1)}\left( \sum_{s_j = s_j'} \sum_{\substack{x^k, y^k : \\ s_j x^k = s_j' y^k}} E(x^k) E(y^k) + \sum_{s_j \neq s_j'} \sum_{\substack{x^k, y^k : \\ s_j x^k = s_j' y^k}} E(x^k) E(y^k) \right) \\
&= q^{(d-2)(k-1)}\left( A + B \right)
\end{align*}

Let us approach $A$ first. Since $s_1 = s_1'$, and $s_j x^k = s_j' y^k$ for all $j$, we know that it holds for $j=1$, and therefore $x^k = y^k$. This tells us that $s_j = s_j'$ for all $j$. So

$$
A = \sum_{s_1, ...,s_{k-1}} \sum_{x^k}E(x^k)E(x^k) = q^{(k-1)}\sum_{x^k}E(x^k)E(x^k) = \vert E \vert q^{(k-1)}
$$

Now we tackle the quantity $B$. Here we introduce new variables, $\alpha = \frac{s_1}{s_1'}$ and $\alpha_j' = s_j'$. Also notice that the condition $s_j x^k = s_j' y^k$ implies $\alpha = \frac{s_j}{s_j'}$ for all $j$. So we did have to sum over $2(k-1)$ different variables, but now we know that these are completely determined by only $k$ of the originals. So we will have $(k-2)$ free variables. In light of this, with a simple change of variables we get

\begin{align*}
B &= q^{(k-2)} \sum_{y^k = \alpha x^k} E(x^k) E(\alpha x^k) \\
&\leq q^{(k-2)} \sum_{x^k \in \mathbb{F}_q^d} \vert E \cap l_{x^k} \vert \\
&\leq \vert E \vert q^{(k-1)}
\end{align*}

Where, $l_{x^k} := \left\lbrace tx^k \in \mathbb{F}_q^d : t \in \mathbb{F}_q\right\rbrace$, which can only intersect $E$ at most $q$ times. With the estimates for $A$ and $B$ in tow,

\begin{align*}
\Vert r_{k-1} \Vert_{L^2}^2 &= q^{(d-2)(k-1)}\left( A + B \right) \\
&\leq q^{(d-2)(k-1)} \left( 2 \vert E \vert q^{(k-1)}\right) \\
&= 2 \vert E \vert q^{(d-1)(k-1)}.
\end{align*}

\end{proof}

\section{Sharpness examples}

\indent The following lemmata are included to show how close Theorem \ref{main} is to being sharp. There are several possible notions of sharpness for this result. It is clearly interesting to consider how big a set can be without containing \textit{any} orthogonal $k$-tuples. The first lemma is merely an instuitive construction used in the next lemma, both of which concern large sets with no orthogonal $k$-tuples. The last lemma is included to show how close our size condition is if we relax the allowable number of orthogonal $k$-tuples.\\

\begin{lemma}\label{d=2}
There exists a set $E \subset \mathbb{F}_q^2$ such that $\vert E \vert \approx q^2$, but no pair of its vectors are orthogonal. 
\end{lemma}
\begin{proof}
This is done by taking the union of about $\frac{q}{2}$ lines through the origin, such that no two lines are perpendicular, and removing the union of their $\frac{q}{2}$ orthogonal complements, which are lines perpendicular to lines in the first union. Then our set $E$ has about $\frac{q^2}{2}$ points, but no pair has a zero dot product.
\end{proof}

The next result is the main counterexample, which shows that it is possible to construct large subsets of $\Bbb F_q^d$ with no pairs of orthogonal vectors.

\begin{lemma}\label{k=2}
There exists a set $E \subset \mathbb{F}_q^d$ such that $\vert E \ge cq^{\frac{d}{2}+1}$, for some $c>0$, but no pair of its vectors are orthogonal. 
\end{lemma}
\begin{proof}
The basic idea is to construct two sets, $E_1 \subset \mathbb{F}_q^2$, and $E_2 \subset \mathbb{F}_q^{d-2}$, such that $\vert E_1 \vert \approx q^\frac{3}{2}$ and $\vert E_2 \vert \approx q^\frac{d-1}{2}$. If you pick $q$ and build these sets carefully, you can guarantee that the sum set of their respective dot product sets does not contain $0$. The following algorithm was inspired by \cite{HIKR07}. \\
\indent Here we will indicate how to construct $E_1$. The construction of $E_2$ is similar. First, let $q = p^2$, where $p$ is a power of a large prime. We also pick these such that $p+1$ is of the form $4n$, where $n$ is odd. This way we can be guaranteed a large, well-behaved multiplicative group of order $q-1=(p-1)(p+1)$, as well as a subfield of order $p$.\\
\indent Let $i$ denote the square root of $-1$, which is in $\Bbb F_q^*$, since $q$ is congruent to 1 mod 4. Now let $B$ be a cyclic subgroup of $\Bbb F_q^*$ of order $\frac{p+1}{4}(p-1) = n(p-1)$. Since $n$ was odd, and $p$ was congruent to 3 mod 4, we know that 4 does not divide the order of $B$. This means that $B$ has no element of order 4, so it is clear that $i \notin B$. Let $\beta$ denote the generator of $B$, as it is a subgroup of a cyclic group, and therefore cyclic. Since $p-1$ is even, we know that we can find another cyclic subgroup, $A$, generated by $\beta^2$. Let $C_p$ be the elements of $\Bbb F_p^*$ that lie on the unit circle, that is,

$$
C_p := \left\lbrace x \in \Bbb F_p^2 : x_1^2 + x_2^2 = 1 \right\rbrace.
$$

\indent From a lemma in \cite{HIKR07}, (or basic number theory) we know that $\vert C_p \vert = p - 1$, since $-1$ is not a square in a field of order congruent to $3$ mod $4$. We can be sure that for all $u,v \in C_p, u \cdot v \in \Bbb F_p$. Now let
$$
E_1' := \left\lbrace \tau u : \tau \in A, u \in C_p \right\rbrace.
$$

So, for all $x,y \in E_1'$, we can be sure that $x \cdot y \in A \cup \lbrace 0 \rbrace$. To see this, let $x = \sigma u$, and $y = \tau v$, where $\sigma, \tau \in A$ and $u,v \in C_p$. Then $x \cdot y = \sigma\tau(u \cdot v) \in A \cup \left\lbrace 0 \right\rbrace$, as any non-zero $u \cdot v \in \Bbb F_p^* \subset A$. Now, the cardinality of $E_1'$ is

$$
\vert E_1' \vert = \vert C_p \vert \vert A \vert = (p-1) \left( \frac{p+1}{4} \frac{p-1}{2}\right) \approx q^\frac{3}{2}.
$$

\indent Now pick $\frac{q}{2}$ mutually non-orthogonal lines in $E_1'$. Call this collection of lines $L$. Let $L^\perp$ indicate the set of lines perpendicular to the lines in $L$. Now we need to prune $E_1'$ so that it has no orthogonal vectors. One of the sets $E_1' \cap L$ or $ E_1' \cap L^\perp$ has more points. Call the set with more points $E_1$. This means that no zero dot products can show up in $E_1$, in a similar manner to the construction in the proof of Lemma \ref{d=2}. Now we have $\vert E_1 \vert \approx q^\frac{3}{2}$, and for any $x,y \in E_1$, we are guaranteed that $x \cdot y \in A$, which does not contain 0.\\
\indent Construct $E_2 \subset \Bbb F_q^{d-2}$ in a similar manner, using spheres instead of circles. However, in the construction of $E_2$, we do not need to prune anything. Now we have $\vert E_2 \vert \approx q^\frac{d-1}{2}$ and all of its dot products lie in $A \cup \lbrace 0 \rbrace$. Set $E = E_1 \times E_2$. Since $E_1$ has its dot product set contained in $A$, and $E_2$ has its dot product set contained in $A \cup \lbrace 0 \rbrace$, we know that any dot product of two elements in $E$ is in the sum set $A + \left( A \cup \lbrace 0 \rbrace \right)$.\\
\indent Now we will show that 0 is not in the dot product set. If two elements did have a zero dot product, that would mean that we had $s,t \in A$, where $s$ comes from the first two dimensions, or $E_1$, and $t$ comes from the other $d-2$ dimensions, or $E_2$, and we also have $s=-t$. (Note, even though $t$ could conceivably be zero, $s$ can not, so we would not have $s=-t$ if $t$ were zero. Therefore $t$ is necessarily an element of $A$.) Recall that $s$ and $t$ are squares of elements in $B$. Call them $\sigma^2$ and $\tau^2$, respectively, for some $\sigma, \tau \in B$. Since $B$ has multiplicative inverses, let $\alpha = \frac{\sigma}{\tau} \in B$. So we would need the following:

$$
\sigma^2 = -\tau^2 \Rightarrow -1 = \frac{\sigma^2}{\tau^2} = \alpha^2.
$$

But we constructed $B$ so that it does not contain the square root of $-1$. Therefore there can be no two elements of $E$ which have a zero dot product.
\end{proof}

\indent The authors believe that the preceeding example can be generalized to obtain results about how large a set can be without containing orthogonal $k$-tuples for $k>2$. The next example is trivial, but is included to indicate one way in which we get some orthogonal vectors, but not as many as would be expected by initial considerations.

\begin{lemma}\label{d-1}
There exists a set $E \subset \mathbb{F}_q^d$ such that $\vert E \vert \approx q^{\frac{k-1}{k}(d-1) + \frac{k-1}{2} + \frac{1}{k}}$, but only $q^{(k-1)(d-1) + 1}$ $k$-tuples of its vectors are orthogonal. 
\end{lemma}
\begin{proof}
A more interesting example is for any suitable $d$ and $k$, consider a set $E_1 \subset \mathbb{F}_q^{d-1}$ that satisfies the size condition for $d-1$ dimensions. That is, $\vert E_1 \vert \approx q^{\frac{k-1}{k}(d-1) + \frac{k-1}{2} + \frac{1}{k}}$. By Theorem \ref{main}, we know that $E_1$ has about $\vert E \vert^kq^{-\left(^k_2\right)} \approx q^{(k-1)(d-1) + 1}$ different $k$-tuples of mutually orthogonal vectors. If we consider $E = E_1 \times \lbrace 0 \rbrace \subset \mathbb{F}_q^d$, we have $\vert E \vert = \vert E_1 \vert \approx q^{\frac{k-1}{k}(d-1) + \frac{k-1}{2} + \frac{1}{k}}$, but we have only about $q^{(k-1)(d-1) + 1}$ different $k$-tuples of mutually orthogonal vectors.
\end{proof}

\indent The set constructed in Lemma \ref{d-1} has fewer than the ``statistically expected" number of mutually orthogonal vectors with respect to the dimension $d$. Considering only the size of the set, this example shows that if we lose a factor of about $q^\frac{k-1}{k}$ points in our set, we lose a factor of about $q^{k-1}$ $k$-tuples of mutually orthogonal vectors. Of course, this is quite artificial, in the sense that $E$ ``really" lives in a smaller dimension, but it does indicate that the result is relatively sharp. This also points to a rather intuitive measure of dimension of a set, if one has a handle on the number of mutually orthogonal $k$-tuples.

\end{document}